\documentclass[11pt]{article}
\usepackage[a4paper,margin=1in,footskip=0.25in]{geometry}
\pdfoutput=1
\usepackage{verbatim}
\usepackage[utf8]{inputenc}
\usepackage{amsmath,amssymb,amsthm}
\usepackage{tikz-cd}
\usepackage{pgfplots}
\pgfplotsset{compat=1.11}
\usepgfplotslibrary{fillbetween}
\usetikzlibrary{intersections}
\usepackage{bbm}
\usepackage{hyperref}
\usepackage{comment}
\usepackage{stmaryrd}
\usepackage[new]{old-arrows}
\usepackage{booktabs}
\usepackage{graphicx}
\graphicspath{ {./Images/} }
\usepackage{todonotes}

\usepackage [english]{babel}
\usepackage [autostyle, english = american]{csquotes}
\MakeOuterQuote{"}

\numberwithin{equation}{section}
\newcommand{\emb}[1]{\textcolor{black}{#1}}

\theoremstyle{definition}
\newtheorem*{definition}{Definition}
\newtheorem{remark}{Remark}

\newtheorem*{main theorem}{Main Theorem}

\newtheorem{problem}{Problem}

\newtheorem{lemma}{Lemma}
\newtheorem{theorem}{Theorem}
\newtheorem{corollary}{Corollary}
\newtheorem{conjecture}{Conjecture}

\newtheorem{appendixlemma}{Lemma}[section]
\newtheorem{remarksection}{Remark}[section]

\newtheorem{theoremsection}{Theorem}[section]

\pgfdeclarelayer{bg}
\pgfsetlayers{bg,main}

\begin{document}
\title{Regularity of Non-stationary Stable Foliations \\ of Toral Anosov Maps}
\author{\bf Alexandro Luna}
\date{}
\maketitle
\begin{abstract}
We consider a sequence of $C^2$ (or $C^3$) Anosov maps of the two-dimensional torus that satisfy a common cone condition, and show that if their $C^2$ (respectively, $C^3$) norms are uniformly bounded, then the non-stationary stable foliation must be of class $C^1$ (respectively, $C^{1+\text{H\"older}}$). This generalizes the classical results on smoothness of the invariant foliations of Anosov maps. We also provide an example that shows that an assumption on boundedness of the norms cannot be removed, which is a phenomenon that does not have an analog in the stationary setting.

The main motivation stems from a standing conjecture concerning the dimension properties of the spectra of Sturmian Hamiltonian operators, and this result serves as a first step towards addressing this conjecture. A detailed appendix is provided showing the potential argument and connection between this theory of non-stationary hyperbolic dynamics and the spectral dimension of these operators.

We also provide an addendum demonstrating that a similar result holds for a sequence of Anosov maps of the $d$-dimensional torus whose stable directions have codimension $1$.
\end{abstract}



\section{Introduction}

It is a classical result that  the stable foliation of a $C^2$ Anosov diffeomorphism of a compact Riemannian surface is $C^1$ (Theorem 6.3 \cite{hp}). If the diffeomorphism is $C^3$, then the foliation can even be taken to be $C^{1+\text{H\"older}}$ (Remark on p.38 \cite{HPS}, see also Appendix 1 in \cite{pt2}). 

 These, and more general, results turned out to be extremely useful in smooth ergodic theory for Hopf-type arguments \cite{A1967, H1971}, and also in rigidity theory \cite{HK1990, G2001}. Moreover, holonomies along invariant foliations do not change (too much) the fractal dimension of sets, and this fact was heavily used in dynamical systems \cite{PV1988}, and also in spectral theory, via the study of the dynamics of the so-called trace maps \cite{dg3, M}. 

 These applications motivated several studies of the regularity properties of these invariant foliations, starting with classical monographs such as \cite{HPS} and \cite{S1987}. In particular, it  was studied in \cite{H, H1997} for Anosov flows, and in \cite{PSW} for partially hyperbolic systems. For volume preserving Anosov flows over manifolds of dimension three, it is known that these foliations, and even the direction fields of their tangent planes, are $C^1$ with derivative of  Zygmund class \cite{HK1990, FH2003}. Moreover, if the flow is smooth, and these foliations are at least  $C^{1,\text{Lipschitz}}$, then the flow is necessarily smoothly conjugate to a linear Anosov flow \cite{HK1990}, which represents another rigidity-type result. In \cite{G2001}, a similar result was given for volume preserving Anosov diffeomorphisms of the two-torus. Namely, it is shown that in this setting, the holonomy of the unstable foliation has a derivative of bounded variation if and only if the Anosov diffeomorphism is smoothly conjugate to a linear map. In \cite{PR2005}, a rigidity result regarding the stable/unstable dimensions of hyperbolic invariant sets and regularity of stable/unstable holonomies was shown; see also the monograph \cite{PRF2009} for a detailed exposition of related results.

In this paper we provide a different kind of generalization of the classical results. Namely, we derive results on the regularity of stable foliations in the non-stationary case for families of Anosov diffeomorphisms of the two-torus. This setting is certainly not new. Many questions regarding (non-stationary,  non-autonomous) Anosov families have been studied staring with  \cite{AF}. In \cite{Mu1}, openness of the set of Anosov families in the space of two-sided sequences of $C^1$ diffeomorphisms, equipped with Whitney topology, was proved. Structural stability results were derived in \cite{Mu2, CRV}, and existence of local stable and unstable manifolds was shown in \cite{Mu}. Regularity properties of these stable/unstable foliations were studied in \cite{S}, but only in the situation where the tail end of the sequence of maps being considered is the same map. Other notions related to hyperbolicty such as Markov partitions, expansiveness, and a shadowing property for Anosov families were addressed in \cite{MuR}.

For random hyperbolic dynamics, much of the invariant manifold theory for discrete-time uniformly hyperbolic systems has been extended to this setting (see \cite{Ar, GK} and references therein). In \cite{ZLZ}, regularity of random stable foliations in the neighborhood of a common fixed point among the maps was examined in order to derive a random version of Belitskii's (also known as Sternberg's) $C^1$ Linearization Theorem.  

Here, we are interested in the regularity of the stable foliations. Our main result is given as follows:

If $\overline f=(f_n)$ is a sequence of $C^2$ diffeomorphisms of $\mathbb T^2$ whose elements satisfy a common cone condition, then it follows that the \textit{non-stationary stable set} with respect to $\overline f$ given by 

\begin{equation}\label{non stationary stable set}
    W^s\left(x,\overline f\right):=\left\{y\in\mathbb{T}^2: \lim\limits_{n\rightarrow\infty}\left|f_{n}\circ\cdots\circ f_{1}(x)-f_{n}\circ\cdots\circ f_1(y)\right|=0\right\}
\end{equation}
is a $C^2$ one-dimensional curve \cite{Ar}. Denoting $\mathcal W^s\left(\overline f \right):=\left\{W^s\left(x,\overline f\right)\right\}_{x\in\mathbb T^2}$, we prove

\begin{theorem}\label{Main Foliation Theorem}
    If $\overline  f=(f_n)$ is a sequence of $C^2$ diffeomorphisms of $\mathbb T^2$ such that $\{f_n\}$ satisfies a common cone condition and the collection of maps, and their inverses, are uniformly bounded in the $C^2$ topology, then  $\mathcal W^s\left(\overline f\right)$ forms a $C^1$ foliation of $\mathbb T^2$. 
    
    If $\{f_n\}$ is a collection of $C^3$ diffeomorphisms and the collection of maps, and their inverses, is uniformly bounded in the $C^3$ topology, then the foliation is $C^{1+\text{H\"older}}$.
\end{theorem}

 When $r\geq 1$, a standard technique in the classical setting to prove that invariant  foliations are $C^r$ has been to use a version of the $C^r$ section theorem to prove that the associated (strong) stable/unstable splittings are $C^r$, and hence must uniquely integrate to $C^r$ foliations. Many versions of this section theorem have been derived in the literature and a generalized version can be found in \cite{PS2}. While our approach is similar, and we prove smoothness of the non-stationary stable direction field (see Theorem \ref{Smoothness of Splitting} below), our proof requires uniform boundedness assumptions on the $C^2$ (or even $C^3)$ norms of the maps, which is a dilemma that does not arise in the stationary setting. We even construct an example to demonstrate that this assumption cannot be removed. To the best of our knowledge, such phenomenon has not been considered before.
 
As seen in the discussion of regularity of foliations in \cite{PSW}, the claim that these foliations have $C^r$ tangent plane fields is stronger than the statement that the foliations are $C^r$ themselves. The original assumption to prove $C^{1+\text{H\"older}}$ regularity of invariant foliations was $C^{2+\text{H\"older}}$ smoothness of the diffeomorphism (see \cite{hp,HPS}). That assumption was relaxed in \cite{PR2002, T}. There, the authors prove $C^{1+\text{H\"older}}$ regularity of invariant foliations for a hyperbolic invariant set of a $C^{1+\text{H\"older}}$  diffeomorphism, under the assumption that this invariant set has local product structure and the stable leaves are one-dimensional. \emb{Here, the authors work with the the basic holonomy maps directly, and do not claim any type of regularity for the stable/unstable directions fields (i.e. tangent planes of the stable/unstable laminations).} For the applications we have in mind (see Appendix \ref{Applications}), we are interested in the dynamics of real analytic (in fact, polynomial) maps, so we do not generalize this particular result. \emb{It would be interesting to see whether a similar result holds in the non-stationary setting, but the techniques do not directly transfer since the works \cite{PR2002, T}  directly rely on the stationary setting. In particular, they make use of both the stable and unstable manifolds, but for our non-stationary setting, only the notion of stable manifolds make sense.}

In Section \ref{Torus Section Problem} we derive a non-stationary version of the $C^{r}$ section theorem, for \emb{$r\in [1,2)$.} In Section \ref{Proof} we use it to  prove Theorem \ref{Main Foliation Theorem} by  showing  that the non-stationary stable direction fields are $C^1$, or $C^{1+\text{H\"older}}$, under the given assumptions \emb{(}Theorem \ref{Smoothness of Splitting}\emb{)}. Finally, in Section \ref{Non-example section}, we construct an example demonstrating that the uniform $C^2$ boundedness assumption of the maps cannot be removed. \emb{In Section \ref{s: The Codimension 1 Case}, as an addendum, we state and sketch the proof of an improved version of this theorem for the case when each $f_n$ has stable directions with codimension $1$. In order to convey a clear argument, and since our future application is very specific to surface diffeomorphisms, we give full details for Theorem \ref{Main Foliation Theorem} as stated above and later explain how a nearly identical proof gives the stable codimension 1 version.}

The initial motivation for this project stems from attempts to analyze the non-stationary dynamics of the so-called \textit{trace maps}, as the dynamics of these maps have played a central role in understanding the dimensional properties of the spectra of one-dimensional Sturmian Hamiltonians. \emb{That is, we consider the bounded self-adjoint operator $H_{\lambda, \alpha, \omega}: \ell^2(\mathbb{Z})\rightarrow \ell^2(\mathbb{Z}) $
via
$$[H_{\lambda, \alpha, \omega}u](n)= u(n+1)+u(n-1)+\lambda\chi_{[1-\alpha, 1)}\left(\omega+n\alpha \ (\text{mod} \ 1)\right)u(n),$$
where $\lambda>0$, $\alpha\in(0,1)$ is irrational, and $\omega\in S^1$. Such operators have been a leading choice of model in studying electronic properties of one-dimensional quasicrystals. It is known (e.g. see Theorem 4.9  \cite{DF}) that the spectrum of $H_{\lambda, \alpha, \omega}$ is independent of $\omega$ and that such a spectrum is a Cantor set of Lebesgue measure zero \cite{bist}. The following conjecture was posed by Bellissard in the 1980s.
\begin{conjecture}\label{main conjecture}
For each $\lambda>0$, the Hausdorff dimension of the spectrum of $H_{\lambda,\alpha,\omega}$ is $\alpha$-constant Lebesgue almost everywhere. 
\end{conjecture}\label{c: almost sure constancy}
The conjecture has been proved when $\lambda\geq 24$ \cite{DG2015} but not even partial results are known when $\lambda$ is small. In Appendix \ref{Applications}, we explain how such spectra can be derived from the non-stationary dynamics of a given family of surface hyperbolic polynomial maps (the trace maps), and in particular, explain how the related smooth foliation theory can be useful in tackling Conjecture \ref{main conjecture}.}

\subsection{Background and Preliminaries}\label{background and preliminaries}
Let $\mathcal F\subset \text{Diff}^1\left(\mathbb T^2\right)$. We say that $\mathcal F$ satisfies a \textit{common cone condition} if there are constants $0<\mu<1<\eta$, such that for each $x\in \mathbb T^2$, there are one-dimensional subspaces $H_x$ and $V_x$, that depend continuously on $x$, such that $T_x\mathbb T^2=H_x\oplus V_x$, so that for each $f\in\mathcal F$, we have

\begin{itemize}
   
        \item[1.)] For the sets
      $$K^s(x):=\{ (v_1, v_2)\in H_x\oplus V_x : |v_1| \leq \mu |v_2|\}, $$ $$ K^u(x):=\{ (v_1, v_2) \in H_x\oplus V_x: |v_2| \leq \mu |v_1|\},$$
       we have 
       \begin{equation}\label{invariance of the cone field}
           {D_xf \left(K^u(x)\right) } \subseteq \text{int\,}K^u(f(x)) \ \ \text{and} \ \ {D_xf^{-1} \left(K^s(x)\right)}\subseteq \text{int\,}K^s\left (f^{-1}(x) \right),
       \end{equation}
       and the angle between any vector from $D_xf \left(K^u(x)\right)$ and $\partial K^u(f(x))$, and the angle between any vector from ${D_xf^{-1} \left(K^s(x)\right)}$ and $\partial K^s\left(f^{-1}(x)\right)$, are uniformly bounded away from zero
      \item[2.)] $\left \|D_xf v\right\| \geq \eta \|v\|$ for $v\in K^u(x)$
    \item[3.)] $\left \|D_xf^{-1}v \right \| \geq \eta \|v\|$ for $v\in K^s( x)$
\end{itemize}
The sets $K^s(x)$ (resp. $K^u(x)$) are usually referred to as \textit{stable} (resp. \textit{unstable}) \textit{cones}. 
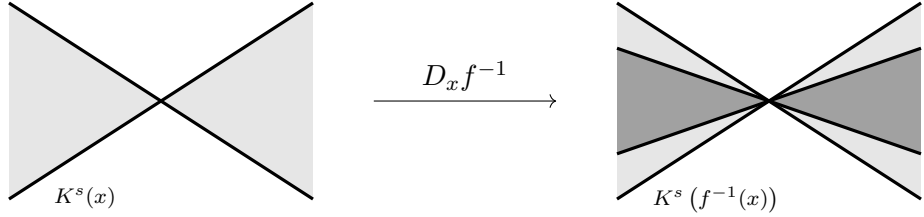
\begin{figure}[h]
   \begin{center}
    
   \begin{tikzpicture}

\draw[black, very thick] (6, .65*2)-- (2,.65*-2);
\draw[black, very thick] (6, -.65*2)-- (2,-.65*-2);

\draw[black, very thick] (6, .35*2)-- (2,.35*-2);
\draw[black, very thick] (6, -.35*2)-- (2,-.35*-2);

\fill[black, opacity=0.1] (6, .65*2) -- (4,0) -- (6, -.65*2)-- (6, .65*2);

\fill[black, opacity=0.1] (2,.65*-2) -- (4,0) -- (2,-.65*-2) -- (2,.65*-2);

\fill[black, opacity=0.3] (6, .35*2) -- (4,0) -- (6, -.35*2)-- (6, .35*2);

\fill[black, opacity=0.3] (2,.35*-2) -- (4,0) -- (2,-.35*-2) -- (2,.35*-2);

\draw[black, very thick] (-6, .65*2)-- (-2,.65*-2);
\draw[black, very thick] (-6, .65*-2)-- (-2,.65*2);

\fill[black, opacity=0.1] (-6, .65*2) -- (-4,0) -- (-6, -.65*2) -- (-6, .65*2);

\fill[black, opacity=0.1] (-2, .65*2) -- (-4,0) -- (-2, -.65*2) -- (-2, .65*2);

\draw[->]  (-1.2,0) -- (1.2,0) node[midway, above]{$D_xf^{-1}$};

\coordinate[label=below: \scriptsize $K^s(x)$] (C) at (-5,-1);

\coordinate[label=below: \scriptsize $K^s\left(f^{-1}(x)\right)$] (C) at (3.3,-1);
\end{tikzpicture}

    \caption{Image of the cone $K^s(x)$ under $D_xf^{-1}$.}
    \label{Cone condition figure}

    \end{center}
\end{figure}

Let $\overline f=(f_n)$ be a sequence in $\mathcal F$. For each $x\in\mathbb{T}^2$, the stable set $W^s\left(x, \overline f\right)$ as defined in (\ref{non stationary stable set}) is the image of an injective $C^1$ immersion from $\mathbb{R}$ to $\mathbb T^2$. This can be seen as a result from random hyperbolic dynamics \cite{Ar} or from the related results in \cite{Mu}, and we also have that $\mathcal W^s\left(\overline f \right):=\left\{W^s\left(x,\overline f\right)\right\}_{x\in\mathbb T^2}$ is a collection of pair-wise disjoint curves that form a continuous foliation of $\mathbb T^2$. From this same theory, it is also known that

\begin{lemma}
    For each $x\in \mathbb T^2$, we have that 
    \begin{equation}\label{stable direction as pullback of cones}
        E^s(x,\overline f):=\bigcap_{n=1}^{\infty}D_{f_n\circ\cdots\circ f_1(x)}\left(f_1^{-1}\circ\cdots\circ f_n^{-1}\right)\left[K^s\left(f_n\circ\cdots\circ f_1(x)\right)\right]
    \end{equation}
    is a one dimension subspace of $T_x\mathbb T^2$ that varies continuously in $x$. Moreover, $T_xW^s\left(x,\overline f\right)=E^s(x,\overline f)$.
\end{lemma}
\begin{proof}
    That $E^s\left(x,\overline f\right)$ is one-dimensional and continuous in $x$ follows from Proposition 6.2.12 and Lemma 6.2.15, respectively, in \cite{KH}. That $T_xW^s\left(x,\overline f\right)=E^s(x,\overline f)$ follows immediately from the existence of $W^s_{\text{loc}}\left(x,\overline f\right)$ as a $C^1$ curve.
\end{proof}
 Define $E^s\left(\overline f\right):=\bigsqcup_{x\in\mathbb T^2}E^s\left(x,\overline f\right),$ to be the \textit{non-stationary stable bundle of} $\overline f$. Notice that (\ref{stable direction as pullback of cones}) is equivalent to say
 \begin{equation}\label{stable bundle as pullback of cone field}
     E^s\left(\overline f \right)=\bigcap_{n=1}^{\infty} T\left(f_1^{-1}\circ\cdots \circ f_n^{-1}\right)\left(K^s\right)
 \end{equation}
 where $K^s:=\bigsqcup_{x\in\mathbb T^2} K^s(x)$, and $Tf$ is the tangential of $f$ defined by $Tf(v)= D_xf(v)$ where $x=\pi(v)$ and $\pi$ is the projection map of the bundle $T\mathbb T^2$ onto $\mathbb T^2$. We will prove

\begin{theorem}\label{Smoothness of Splitting} If $\overline f=(f_n)$, $f_n\in \mathrm{Diff}^2\left(\mathbb T^2\right)$, is a sequence such that $\{f_n\}$ satisfies a common cone condition and 
\begin{equation}\label{C^2 norms bounded condition}
    \max\left\{\sup_{n}\left\|f_n\right\|_{C^2}, \sup_{n}\left\|f_n^{-1}\right\|_{C^2}\right\}<\infty,
\end{equation}
then the bundle $E^s\left(\overline f\right)$ is $C^1$. If $f_n\in \mathrm{Diff}^3\left(\mathbb T^2\right)$ for all $n\in\mathbb N$, and 
\begin{equation}\label{C^3 norms bounded condition}
    \max\left\{\sup_{n}\left\|f_n\right\|_{C^3}, \sup_{n}\left\|f_n^{-1}\right\|_{C^3}\right\}<\infty,
\end{equation}
then the bundle is $C^{1+\beta}$ for some $\beta\in (0,1)$.
\end{theorem}
As a corollary (see Table 1 in \cite{PSW}), we obtain Theorem \ref{Main Foliation Theorem}.

\begin{remark}
    As we will see from Remark \ref{torus section alternate holder condition} below, the $C^{1+\text{H\"older}}$ version can be obtained by replacing assumption (\ref{C^3 norms bounded condition}) with the condition that each $f_n$ is $C^{2+\epsilon}$, with common H\"older constant, for some $\epsilon>0$.
\end{remark}

\begin{remark}
    We note that condition (\ref{C^2 norms bounded condition}) is satisfied in the natural cases when $\{f_n\}$ is a finite set, each $f_n$ is a $C^2$ perturbation of a specific $C^2$ Anosov map, or each $f_n$ is $C^{2+\text{H\"older}}$ with a common H\"older constant and exponent.
\end{remark}

\begin{remark}
    For the proof to come, the main tools will mainly rely on the portion of the cone condition involving the stable cones. However, we note that the unstable cone condition is needed to guarantee the existence of the stable manifolds as one dimensional curves. Notice that the full cone conditions also ensures that each individual map $f_n$ is an Anosov diffeomorphism.
\end{remark}

For the convenience of presentation, we give the full proof of this theorem as stated, but later in Section \ref{s: The Codimension 1 Case}, we formulate a more general version in the case when the corresponding stable directions of each $f_n$ have codimension $1$. This amounts to considering Anosov maps $f_n:\mathbb T^d\rightarrow \mathbb T^d$, $d>1$, where they satisfy a common cone condition as above but with
$\dim(H_x)=d-1$ and $\dim(V_x)=1$. 
The proof is nearly identical, but the presentation is cleaner in the surface case, which we believe offers an overall more convincing argument. The reader may skip the latter presentation on the first read.

\section{Non-stationary Section Theorem}\label{Torus Section Problem}
In order to prove Theorem \ref{Smoothness of Splitting}, we will first prove non-stationary versions of the $C^{1}$ and $C^{1+\text{H\"older}}$ section theorems. The proof will rely on the technical contraction mapping lemmas in Appendix \ref{Technical Lemmas}.

For any $r\in\left(0,\frac{1}{2}\right)$ and continuous map $\varphi(x):\mathbb T^2\rightarrow S^1$, denote $K_r(\varphi(x)):= [\varphi(x)-r, \varphi(x)+r]$, and define
$$K(r,\varphi):=\{(x,z)\in \mathbb T^2\times S^1: z\in K_r(\varphi(x))\}.$$
We refer to $K(r,\varphi)$ as a \textit{band}.
\begin{theorem}\label{Torus Section Theorem}
Suppose that $(g_n)$, $g_n:\mathbb T^2\rightarrow \mathbb T^2$, is a  sequence of $C^1$ diffeomorphisms such that
\begin{equation}
  \kappa:=  \sup_{n\in\mathbb{N}}\left\{ \sup_{x\in\mathbb T^2}\kappa_n(x)\right\}<\infty,
\end{equation} 
where $\kappa_n(x)=\left\|D_xg_n^{-1}\right\|$.
Furthermore, suppose $(h_n)$, $h_n:\mathbb T^2\times S^1\rightarrow S^1$, a sequence of $C^1$ maps such that, for some band $K(r,\varphi)$, we have
\begin{equation}\label{Fiber invariance}
    h_n(x,z)\in K_r(\varphi(g_nx))
\end{equation}
for each $(x,z)\in K(r,\varphi)$ and $n\in\mathbb N$, and 
\begin{equation}\label{fiber contraction rate}
   \lambda:=\sup_{n\in\mathbb N} \sup\limits_{(x,z)\in K(r,\varphi)}\left\{\left|\lambda_n(x,z) \right|\right\}<1,
\end{equation}
where $\lambda_n(x,z):=\partial_zh_n(x,z)$.
Set $$F_n:\mathbb T^2\times S^1\rightarrow \mathbb T^2\times S^1, \ (x,z)\mapsto \left(g_nx, h_n(x,z)\right).$$
If $\{h_n\}$ is uniformly bounded in the $C^1$ topology and 
\begin{equation}\label{fiber and base contraction for skew product}
    \Delta:=\sup_{n\in\mathbb N}\left\{\sup\limits_{(x,z)\in K(r,\varphi)}\lambda_n\left(g_n^{-1}x,z\right)\kappa_n(x)\right\}<1,
\end{equation} 
then there is a $C^1$ map $\sigma^*:\mathbb T^2\rightarrow S^1$ such that 

\begin{equation}\label{e: Intersection of cones is graph}
    \bigcap_{n=1}^{\infty}F_1\circ\cdots\circ F_n\left(K(r,\varphi)\right)=\mathrm{graph}(\sigma^*).
\end{equation}
Furthermore, if we assume that the maps $h_n$ are $C^2$, the collection $\{h_n\}$ is uniformly bounded in the $C^2$ topology, and
\begin{equation}\label{holder fiber and base contraction for skew product}
    \Delta_
\beta:=\sup_{n\in\mathbb N}\left\{\sup\limits_{(x,z)\in K(r,\varphi)}\lambda_n\left(g_n^{-1}x,z\right)\kappa_n(x)^{1+\beta}\right\}<1,
\end{equation}
for some $\beta\in(0,1)$, then $\sigma^*$ is $C^{1+\beta}$.
    
\end{theorem}

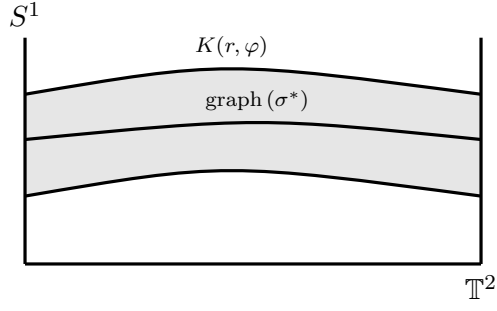
\begin{figure}[h]
   \begin{center}
    
   \begin{tikzpicture}[scale=1.5]

\draw[black, very thick] (-2, 0)-- (-2,2) node[above]{$S^1$};
\draw[black, very thick] (2, 0)-- (2, 2);

\draw[black, very thick] (-2, 0)-- (2, 0) node[below]{$\mathbb T^2$};
\draw[black, very thick] (-2, 1.5) .. controls  (-0.25, 1.8) .. (2,1.5) node[midway, above]{\scriptsize $K(r,\varphi)$};

\draw[black, very thick] (-2, 0.6) .. controls  (-0.25, 0.9) .. (2,0.6);

\draw[black, very thick] (-2, 1.1)  .. controls  (0.05, 1.3) .. (2,1.1) node[midway, above]{\scriptsize $\text{graph}\left(\sigma^*\right)$};

\fill[black, opacity=0.1] 
(-2, 1.5) .. controls  (-0.25, 1.8) .. (2,1.5) -- (2,0.6) .. controls  (-0.25, 0.9) .. (-2, 0.6) -- (-2, 1.5);

\end{tikzpicture}

    \caption{Illustration of Theorem \ref{Torus Section Theorem}}
    \label{Section theorem figure}

    \end{center}
\end{figure}

\begin{proof}
    The proof is broken up into steps:\\
    \\
\noindent\textbf{Step 1: Construction of $\sigma^*$}\\
\\
We note that $X:=\left\{\sigma\in C^0(\mathbb T^2, S^1): \text{graph}(\sigma)\in K(r,\varphi)\right\}$ is a closed and bounded space under the $C^0$ norm. Now, by construction, we have a map $\Gamma_n:X\rightarrow X$ via $\sigma \mapsto h_n\circ (\text{id}, \sigma) \circ g_n^{-1} $. It follows that
\begin{equation}\label{base contraction of graph transform}
    \sup_n\text{Lip}\left(\Gamma_n\right)\leq \lambda,
\end{equation}
so  from (\ref{fiber contraction rate}) and Lemma \ref{non-stationary contraction principle}, there is a unique $\sigma^*$ such that 
$$\lim\limits_{n\rightarrow\infty} \Gamma_1\circ\cdots \circ\Gamma_n(\sigma)=\sigma^*$$
for any $\sigma\in X$. Notice that $F_n\left(\text{graph}(\sigma)\right)=\text{graph}(\Gamma_n\sigma)$ for all $\sigma\in X$ and $n\in\mathbb N$, and hence, (\ref{e: Intersection of cones is graph}) holds.\\
\\
\textbf{Step 2: Lipschitz Continuity of $\sigma^*$}\\
\\
Before showing that $\sigma^*$ is $C^1$, we will first prove that it is Lipschitz, and we will do so by showing that for some $L>0$, the set of maps in $X$ with slope $\leq L$, which is a compact set under the $C^0$ norm, is invariant under $\Gamma_n$ for each $n\in\mathbb N$.

Since $F_n\left(\text{graph}(\sigma)\right)=\text{graph}(\Gamma_n\sigma)$, for all $\sigma\in X$, it will be enough to show that the derivative of each map $F_n$ preserves the horizontal  cone field $\left\{\mathcal K((x,z), L) \right\}_{(x,z)\in\mathbb T^2\times S^1}$ in the tangent bundle $T\left(\mathbb T^2\times S^1\right)$, where
\begin{equation}\label{horizontal cone field in bundle}
    \mathcal K((x,z), L):=\left\{\begin{pmatrix}
v_1\\
v_2
\end{pmatrix}\in T_{(x,z)}\left(\mathbb T^2 \times S^1\right): v_1\in T_x\mathbb T^2, \ v_2\in T_zS^1, \ |v_2|\leq L|v_1|\right\}.
\end{equation} 
Choose $B>0$ so that
\begin{equation}\label{bound for C^1 norm of h_n and g_n}\max\left\{\sup_{n\in\mathbb N}\|h_n\|_{C^1}, \kappa \right\}\leq B.
\end{equation}
Now, choose $L>0$ large enough so that 
$$L>\frac{B^2}{1-\Delta},$$
and so that there exists a Lipschitz function from $\mathbb T^2$ to $S^1$ whose graph is contained in $K(r,\varphi)$ and whose slope is $\leq L$.

Letting $\begin{pmatrix}
v_1\\
v_2
\end{pmatrix}\in \mathcal K((x,z), L)$, we have

$$D_{(x,z)}F_n\begin{pmatrix}
v_1\\
v_2
\end{pmatrix}=\left( \begin{array}{c|c}
   D_xg_n & \boldsymbol{0} \\
   \midrule
   \partial_xh_n(x,z) & \partial_zh_n(x,z) \\
\end{array}\right)\begin{pmatrix}
v_1\\
v_2
\end{pmatrix}=\begin{pmatrix}
D_xg_n(v_1)\\
\partial_xh_n(x,z)(v_1)+\partial_zh_n(x,z)(v_2)
\end{pmatrix}$$
where
$$|v_1|=\left| D_{g_nx}g_n^{-1}\circ D_xg_n(v_1)\right|\leq \kappa_n\left(g_nx\right)\left|D_xg_n(v_1)\right|$$
so that 
\begin{equation}\label{Eqn2}
  \begin{split}
    &|\partial_xh_n(x,z)(v_1)+\partial_zh_n(x,z)(v_2)|\\
    &\leq B|v_1|+\lambda_n(x,z)|v_2|\\
    &\leq \kappa_n(g_nx)B\left|D_xg_n(v_1)\right|+\lambda_n(x,z)L|v_1|\\
    &\leq B^2\left|D_xg_n(v_1)\right|+L\lambda_n(x,z)\kappa\left(g_nx\right)\left|D_xg_n(v_1)\right|\\
    &\leq \left(B^2+\Delta L\right)\left|D_xg_n(v_1)\right|\leq L\left|D_xg_n(v_1)\right|
  \end{split}
\end{equation}
\begin{align*}
    &|\partial_xh_n(x,z)(v_1)+\partial_zh_n(x,z)(v_2)|\leq B|v_1|+\lambda_n(x,z)|v_2|\\
    &\leq \kappa_n(g_nx)B\left|D_xg_n(v_1)\right|+\lambda_n(x,z)L|v_1|\leq B^2\left|D_xg_n(v_1)\right|+L\lambda_n(x,z)\kappa\left(g_nx\right)\left|D_xg_n(v_1)\right|\\
    &\leq \left(B^2+\Delta L\right)\left|D_xg_n(v_1)\right|\leq L\left|D_xg_n(v_1)\right|
\end{align*}
and hence $D_{(x,z)}F_n\begin{pmatrix}
v_1\\
v_2
\end{pmatrix}\in \mathcal K (F_n(x,z), L)$.

Set $X'\subset X$ to be the collection of maps from $X$ whose graphs have slope $\leq L$, which is a compact subset of $X$ in the $C^0$ topology. Now, since $\Gamma_n\left(X'\right)\subset X'$ for all $n\in\mathbb N$, taking $\sigma_0\in X'$ gives that 
$$\lim\limits_{n\rightarrow\infty} \Gamma_1\circ\cdots \circ\Gamma_n(\sigma_0)=\sigma^*\in X'.$$ 
\\
\textbf{Step 3: $C^1$ Regularity of $\sigma^*$}\\
\\
To show that $\sigma^*$ is $C^1$, we will use another graph transform argument to show that its derivative exists and belongs to the space of continuous functions from $\mathbb T^2$ to $L\left(\mathbb R^2, \mathbb R\right)$.

Let $Y=C^0\left(\mathbb T^2, L\left(\mathbb R^2, \mathbb R\right)\right)$ which is a Banach space. For each $n\in\mathbb N$ and $\sigma\in X'$, we define $\Psi_n^\sigma:Y\rightarrow Y$ where for each $H\in Y$ we define
$$\left(\Psi_n^\sigma(H)\right)(x):=D_{\left(g^{-1}_nx, \sigma\left(g^{-1}_nx\right)\right)}h_n\circ\left(\text{Id}, H\left({g_n^{-1}x}\right)\right) \circ D_xg_n^{-1}. $$
Hence, if $\sigma$ is $C^1$, we must have $D(\Gamma_n\sigma)=\Psi_n^{\sigma}(D\sigma)$, and inductively, 
\begin{equation}\label{derivative and graph transform}
    D(\Gamma_k\cdots \Gamma_n\sigma )=\Psi^{\Gamma_{k+1}\cdots \Gamma_n\sigma}_1\circ\cdots\circ \Psi^{\Gamma_n\sigma}_{n-1}\circ \Psi^{\sigma}_n(D\sigma),
\end{equation} 
for all $n,k\in\mathbb N$ with $k\leq n$. 

We will show that Lemma \ref{fiber contraction} applies to the map $\Phi_n: X'\times Y\rightarrow X'\times Y, \ (\sigma, H)\mapsto\left(\Gamma_n\sigma, \Psi^{\sigma}_nH \right)$, and hence there would be a $H^*\in Y$ such that
\begin{equation}\label{Skew product limit}
    \lim\limits_{n\rightarrow\infty} \Phi_1\circ\cdots\circ \Phi_n(\sigma, H)=\lim\limits_{n\rightarrow\infty} \left(\Gamma_1\circ\cdots\circ \Gamma_n \sigma, \Psi^{\Gamma_2\cdots \Gamma_n\sigma}_1\circ\cdots\circ \Psi^{\sigma}_n(H)\right)
     =(\sigma^*, H^*)
\end{equation}
for any $(\sigma,H)\in X'\times Y$. Thus, if $\sigma$ is $C^1$, we have that $D(\Gamma_1\circ\cdot \circ \Gamma_n\sigma)$ converges, and hence $D\sigma^*=H^*$ so that $\sigma^*$ is $C^1$. We now check that $(\Phi_n)$ satisfies the appropriate conditions. 

Condition (\ref{base contraction rate in skew product}) follows from (\ref{base contraction of graph transform}). For each $H,H'\in Y$ and $v\in\mathbb R^2$, we have
\begin{align*}
    &\left(\Psi^{\sigma}_n(H)(x)-\Psi^{\sigma}_n(H')(x)\right)v\\
    &=\begin{pmatrix}
\partial_{x_1}h_n(y, \sigma(y)) \\ \partial_{x_2}h_n(y, \sigma(y)) \\\partial_{z}h_n(y, \sigma(y))
\end{pmatrix}^T	\begin{pmatrix}
0 \\
0 \\
\left[\left(H(y)-H'(y)\right)\circ D_x g_n^{-1}\right](v)
\end{pmatrix}	\\
&=\left[\partial_zh_n(y,\sigma(y))\cdot \left(H(y)-H'(y)\right)\circ D_x g_n^{-1}\right](v),
\end{align*}
where $y=g_n^{-1}x$ and $x=(x_1,x_2)$. Thus, 
\begin{equation}\label{Lipschtiz for Derivative Transform}
   \left\|\Psi_n^{\sigma}(H)(x)-\Psi_n^{\sigma}(H')(x)\right\|\leq \lambda_n(y,\sigma(y))\kappa_n(x)\|H(y)-H'(y)\|,
\end{equation}
and hence condition (\ref{fiber contraction rate in skew product}) follows from (\ref{fiber and base contraction for skew product}). Also, notice that for any $H\in Y$, we have 
$$\Psi_n^{\sigma}(H)(x)-\Psi_n^{\sigma'}(H)(x)=\left(D_{(y, \sigma(y))}h_n-D_{(y, \sigma '(y))}h_n\right)\left((\text{Id}, H)\circ D_{x}g_n^{-1}\right)$$
and so from (\ref{bound for C^1 norm of h_n and g_n}), we have
$$\left\| \Psi_n^{\sigma}(H)(x)-\Psi_n^{\sigma'}(H)(x)\right\|\leq \kappa\left\| D_{(y,\sigma(y))}h_n-D_{(y,\sigma'(y))}h_n\right\|\left\|(\text{Id}, H)\right\|\leq 2B\kappa (1+\|H\|),$$
for any $\sigma, \sigma '\in X'$ and $n\in\mathbb N$ (condition (1) of Lemma \ref{fiber contraction}). In addition, for any $L'>0$, if $H$ belongs to the ball of radius $L'$ in $Y$, we would also have
$$\left\| \Psi_n^{\sigma}(H)(x)-\Psi_n^{\sigma'}(H)(x)\right\|\leq \kappa\left\| D_{(y,\sigma(y))}h_n-D_{(y,\sigma'(y))}h_n\right\|\left(1+L'\right)$$
and since $h_n$ is $C^1$ and $X'$ is compact, we would have that 
$$\lim\limits_{\left\|\sigma-\sigma'\right\|\rightarrow 0} \left\| \Psi^{\sigma}_n(H) -\Psi^{\sigma'}_n(H)\right\|=0,$$
and this limit is uniform in $H\in B_{L'}(0)$ (condition (2) of Lemma \ref{fiber contraction}, also see Remark \ref{equicontinuity replacement remark}).

Since Lemma \ref{fiber contraction} applies to the sequence of maps $(\Phi_n)$, we have that (\ref{Skew product limit}) holds, and hence $\sigma^*$ is $C^1$.\\
\\
\textbf{Step 4: $C^{1+\text{H\"older}}$ Regularity of $\sigma^*$}\\
\\
We now show that $\sigma^*$ is $C^{1+\text{H\"older}}$. Consider the collection of H\"older continuous maps with common constant and exponent $\beta$, which is a closed subspace of $C^0\left(\mathbb T^2, L(\mathbb R^2, \mathbb R)\right)$. The idea is to show that this space is invariant under each $\Phi_n^\sigma$, for any $\sigma\in X'$. This way, assuming $H$ is in this closed space in the limit (\ref{Skew product limit}), we will have that $H^*=D\sigma^*$ is $\beta$-H\"older. We now prove this invariance.

Suppose $\Delta_{\beta}=\sup_{n\in\mathbb N}\left\{\sup\limits_{(x,z)\in K(r,\varphi)}\lambda_n\left(g_n^{-1}x,z\right)\kappa_n(x)^{1+\beta}\right\}<1$ for some $\beta\in(0,1)$. Choose $C>0$ and let $Y'$ be the set of maps $H\in Y$ such that 
$$\limsup_{|x-y|\rightarrow 0}\frac{\|H(x)-H(y)\|}{|x-y|^{\beta}}\leq C$$
which is a topologically closed subspace of $C^0\left(\mathbb T^2, L(\mathbb R^2, \mathbb R)\right)$.

If $\{h_n\}$ is uniformly bounded in the $C^2$ topology, then there is a uniform bound and a uniform Lipschitz constant for the maps $\left\{(x,z)\mapsto D_{(x,z)}h_n \right\}$. Let $M>0$ be a common bound for these quantities.

We will show that if $H\in Y'$, then $\Psi^{\sigma}_n(H)\in Y'$ for all $\sigma\in X'$ and $n\in\mathbb N$. Let $\epsilon>0$ and let $x\in\mathbb T ^2$ be arbitrary. For $x,x'\in \mathbb T^2$, denote $y=g_n^{-1}x$, $y'=g_n^{-1}x'$, $P=D_xg^{-1}_n$, and $P'=D_{x'}g_n^{-1}$. Then, 

\begin{equation*}
  \begin{split}
    \Psi_n^{\sigma}(H(x))-\Psi_n^{\sigma}(H(x'))&=D_{(y,\sigma(y))}h_n\circ (\text{Id}, H(y))P-D_{(y',\sigma(y'))}h_n\circ (\text{Id}, H(y'))P'\\
    &=D_{(y,\sigma(y))}h_n\emb{\circ}\left[(\text{Id},H(y))-(\text{Id},H(y'))\right]P\\
    &\quad +D_{(y,\sigma(y))}h_n\circ (\text{Id}, H(y'))P-D_{(y',\sigma(y'))}h_n\circ (\text{Id}, H(y'))P'.
  \end{split}
\end{equation*}
If $x'$ is sufficiently close to $x$, then
\begin{align*}
    &\left\|D_{(y,\sigma(y))}h_n\emb{\circ}\left[(\text{Id},H(y))-(\text{Id},H(y'))\right]P\right\|\leq |\partial_zh_n(y,\sigma(y))|\cdot\|P\|\cdot\|H(y)-H(y')\|\\
    &\leq (C+\epsilon)\lambda_n(y,\sigma(y))\kappa_n(x)|y-y'|^{\beta},
\end{align*}
and 
\begin{align*}
    &\|D_{(y,\sigma(y))}h_n\circ (\text{Id}, H(y'))P-D_{(y',\sigma(y'))}h_n\circ (\text{Id}, H(y'))P'\|\\
    &\leq \|D_{(y,\sigma(y))}h_n\circ (\text{Id}, H(y'))P-D_{(y',\sigma(y'))}h_n\circ (\text{Id}, H(y'))P\|\\
    &\quad +\|D_{(y',\sigma(y'))}h_n\emb{\circ}(\text{Id}, H(y'))(P-P')\|\\
    &\leq M\cdot |(y,\sigma(y))-(y',\sigma(y')|+\|Dh_n\|\cdot \|(\text{Id}, H(y))\|\cdot \|P-P'\|\\
    &\leq M(\kappa(1+L))|x-x'|+M(1+\|H\|)\kappa |x-x'|.
\end{align*}
Thus, since $1-\beta>0$, we have
\begin{align*}
    &\limsup_{x'\rightarrow x} \frac{\left\|\Psi_n^{\sigma}(H(x))-\Psi_n^{\sigma}(H(x'))\right\|}{|x-x'|^{\beta}}\\
    &\leq \limsup_{x'\rightarrow x}\left((C+\epsilon)\lambda_n(y,\sigma(y))\kappa_n(x)\left(\frac{|y-y'|}{|x-x'|}\right)^{\beta}+M\kappa(2+L+\|H\|)\left|x-x'\right|^{1-\beta}\right)\\
    &\leq (C+\epsilon)\lambda_n(y,\sigma(y))\kappa_n(x)^{1+\beta}.
\end{align*}
Since $\epsilon$ and $x$ were arbitrary and $\Delta_\beta<1,$ we have that $\Psi_n^\sigma(H)\in Y'$.
Hence, choosing a $H\in Y'\cap Y$ in the limit (\ref{Skew product limit}) gives that $H^*=D\sigma^*$ is of class $C^{\beta}$.
\end{proof}

\begin{remark}\label{torus section alternate holder condition}
    To show that $\sigma^*$ is $C^{1+\beta}$, the assumption that $\{h_n\}$ is uniformly bounded in the $C^2$ topology can be replaced with the condition that each $h_n$ is $C^{1+\epsilon}$, with common H\"older constant, for some $\epsilon>\beta$.
\end{remark}
\section{Proof of Main Result}\label{Proof}

Before proving Theorem \ref{Smoothness of Splitting}, we relate it to the setting of Theorem \ref{Torus Section Theorem}. Consider the Grassmanian bundle $\text{Gr}_1\left(T\mathbb T^2\right)$ with projection map $p:\text{Gr}_1\left(T\mathbb T^2\right)\rightarrow \mathbb T^2$. If $E\subset T\mathbb T^2$ is a one-dimensional section, then we will denote $\widehat E\subset G_1\left(T\mathbb T^2\right)$ to be its projective image. If $f$ is a diffeomorphism of $\mathbb T^2$, by an abuse of notation, we will denote $D_xf$, and $Tf$, to be the naturally induced action of the differential, and tangential, of $f$ on $\text{Gr}_1\left(T_x\mathbb T^2\right)$, respectively. 

Each fiber $p^{-1}(x)$ is a copy of $\mathbb R\mathbb P^1$ and is hence diffeomorphic to $S^1$. Let us write this given diffeomorphism by $\phi_x: \text{Gr}_1\left(T_x\mathbb T^2\right)\rightarrow S^1$ and set 
$$\phi:\text{Gr}_1\left(T\mathbb T^2\right)\rightarrow \mathbb T^2\times S^1, \ l\mapsto \left(p(l), \phi_{p(l)}(l) \right).$$
We now prove Theorem \ref{Smoothness of Splitting}.
\begin{proof}[Proof of Theorem \ref{Smoothness of Splitting}] 

Suppose that $\{f_n\}$ satisfies a common cone condition as in Section \ref{background and preliminaries}. We want to show that $\widehat{ E^s\left(\overline f\right)}$ is the image of a $C^1$ section of $\text{Gr}_1\left(T\mathbb T^2\right)$. Let us set $g_n:=f_n^{-1}$ and $h_n:\mathbb T^2\times S^1\rightarrow S^1$ via $$h_n(x,z)= \phi_{f_n^{-1}x}\circ D_xf_n^{-1}\circ \phi_x^{-1}(z),$$
and 
$$F_n:\mathbb T^2\times S^1\rightarrow \mathbb T^2\times S^1, \ (x,z)\mapsto\left( g_n x, h_n(x,z)\right) $$ 
so that
\begin{equation}\label{commutative maps}
    \phi\circ Tf_n^{-1}=F_n\circ \phi
\end{equation}
for all $n\in\mathbb N$.

Notice that if we set $\widehat K^s(x)\subset \text{Gr}_1\left( T_x\mathbb T^2 \right)$ to be the set one dimensional vector spaces contained in $K^s(x)$, then $\phi_x\left( \hat K^s(x)\right)$ is a closed interval in $S^1$ of length less than $\frac{1}{2}$ and this length is independent of $x$, and hence $$K:=\left \{(x,z)\in \mathbb T^2\times S^1: z\in \phi_x\left( \hat K^s(x)\right)\right\} $$ is a band in the sense of Section \ref{Torus Section Problem}. 

 From the fact that each $f_n$ is a $C^2$ diffeomorphism and assuming (\ref{C^2 norms bounded condition}), we have that $\left\{g_n^{-1}\right\}$ is uniformly bounded in the $C^1$ topology, each map $h_n$ is $C^1$, and $\{h_n\}$ is uniformly bounded in the $C^1$ topology. We now check that Theorem \ref{Torus Section Theorem} applies to our sequence of maps $(F_n)$.

From the cone conditions, and the fact that $S^1$ is one-dimensional, we know that the (pointwise) contraction of $h_n\restriction_{K}$ is stronger than the expansion that can be applied by $f_n$ in the base space $\mathbb T^2$. Strictly speaking, we have
\begin{equation}\label{e: Fiber and base contraction}
   \sup_{(x,z)\in K}\left|\partial_zh_n(f_nx, z)\right|\leq   \sup_{(x,z)\in K}\left\|D_xf_n\right\|\cdot \left|\partial_zh_n(f_nx, z)\right|<\frac{1}{\eta}
 \end{equation}
for each $(x,z)\in K$ so that conditions (\ref{fiber contraction rate}) and (\ref{fiber and base contraction for skew product}) hold since $\eta>1$. Condition (\ref{Fiber invariance}) follows from the cone field invariance (\ref{invariance of the cone field}). Thus, by applying Theorem \ref{Torus Section Theorem}, we have that there is a $C^1$ map $\sigma^*:\mathbb T^2\rightarrow S^1$ such that
$$\bigcap_{n=1}^{\infty}F_1\circ \cdots \circ F_n(K)=\text{graph}\left(\sigma^*\right). $$
It follows from (\ref{stable bundle as pullback of cone field}) and (\ref{commutative maps}) that 
$$\widehat{E^s\left(\overline f\right)}=\phi^{-1}\left(\text{graph}\left(\sigma^*\right)  \right),$$
 which proves the $C^1$ case. Notice that if we choose $\beta>0$ so small that 
$$k^{\beta}\eta^{-1}<1$$
where $\kappa:=\sup_{n\in\mathbb N}\sup_{x\in\mathbb T^2}\left\|D_xg_n^{-1}\right\|$, then we obtain
$$\sup_{(x,z)\in K}\left\|D_xf_n\right\|^{1+\beta}\cdot \left|\partial_zh_n(f_nx, z)\right|<\frac{\kappa^{\beta}}{\eta}, $$
so that condition (\ref{holder fiber and base contraction for skew product}) holds. 
Further assuming that each $f_n$ is a $C^3$ diffeomorphism and (\ref{C^3 norms bounded condition}), implies that $\sigma^*$ is $C^{1+\beta}$ by another application of Theorem \ref{Torus Section Theorem}. 
\end{proof}

\begin{remark}
    By Remark \ref{torus section alternate holder condition}, the $C^{1+\text{H\"older}}$ version can be obtained by replacing assumption (\ref{C^3 norms bounded condition}) with the condition that each $f_n$ is $C^{2+\epsilon}$, with common H\"older constant, for some $\epsilon>0$.
\end{remark}
\section{Non-$C^1$ Stable Foliations}\label{Non-example section}
We now construct an example of a sequence $\overline f=(f_n)$, $f_n\in \text{Diff}^2\left(\mathbb T^2 \right)$, such that $\{f_n\}$ satisfies a common cone condition but whose $C^2$ norms are not uniformly bounded, and such that the stable foliation $\mathcal W^s\left(\overline f \right)$ is not $C^1$. 

The example is related to the techniques in \cite{PPR}, where it is proved that most $C^1$ Anosov diffeomorphisms have non-differentiable stable and unstable foliations. We sketch the construction as follows: 

Let $A:\mathbb T^2\rightarrow \mathbb T^2$ be a volume preserving linear hyperbolic map, with unstable eigenvalue $\eta>1$, whose stable and unstable directions are orthogonal. Then, $A$ satisfies a cone condition with respect to cone fields $ K^s:=\left\{K^s(x)\right\}_{x\in\mathbb T^2}$ and $K^u:=\left\{K^u(x)\right\}_{x\in\mathbb T^2}$ , where each stable cone $K^s(x)$ has a small interior angle $\theta$ and whose center line is $E^s(x)$, and $K^u(x)=\left(K^s(x)\right)^{\perp}\emb{:=\left\{v\in T_x\mathbb T^2: v\perp w \ \text{for some} \  w\in K^s(x)\right\}}$.

Consider the local stable manifold $W^s_{\text{loc}}(0,A)$, which is a line through $0$, and two perpendicular lines $\tau$ and $\tau'$ that intersect this line uniquely. Let us set $p:=\tau\cap W^s_{\text{loc}}(0,A)$, $p':=\tau'\cap W^s_{\text{loc}}(0,A),$ $p_n=A^np$, and $p_n'=A^np'$. Then, we know
$$\left|p_n-p_n'\right|=\eta^{-n}\left|p-p'\right|.$$
For each $n\in\mathbb N$, consider a rectangle $R_n$ with center $A^np$ whose height (length of sides parallel to $E^u$) is the same as $\left|p_n-p_n'\right|$ but whose  base (length of sides parallel to $E^s$) is $\frac{\left|p_n-p_n'\right|}{100}$.

We will also consider a rectangle $R'_n$ with center $p_n'$ whose height is $\eta^{-n}\left(\left|p-p'\right|+0.1\right)$ and whose base is also $\frac{\left|p_n-p_n'\right|}{100}$. If we choose $\theta$ small enough, then we can ensure that the following two properties hold:
\begin{itemize}
    \item[(P1)] The cones $K^u(p_n)\cap R_n$ and $K^u\left(p_n'\right)\cap R'_n$ fit inside of $R_n$ and $R_n'$, respectively, as in Figure \ref{Non-example Figure 1}
    \item[(P2)] For any $C^1$ foliation $\mathcal F$ tangent to $K^s$ we have the following: If $\kappa\subset R_n$ and $\kappa'$ are two curves containing $p_n$ and $p_n'$, respectively, and these curves are both tangent to $K^u$, then the holonomy map $H:\kappa \rightarrow \kappa'$, with respect to $\mathcal F$, satisfies $H(\kappa)\subset R_n'$
\end{itemize}

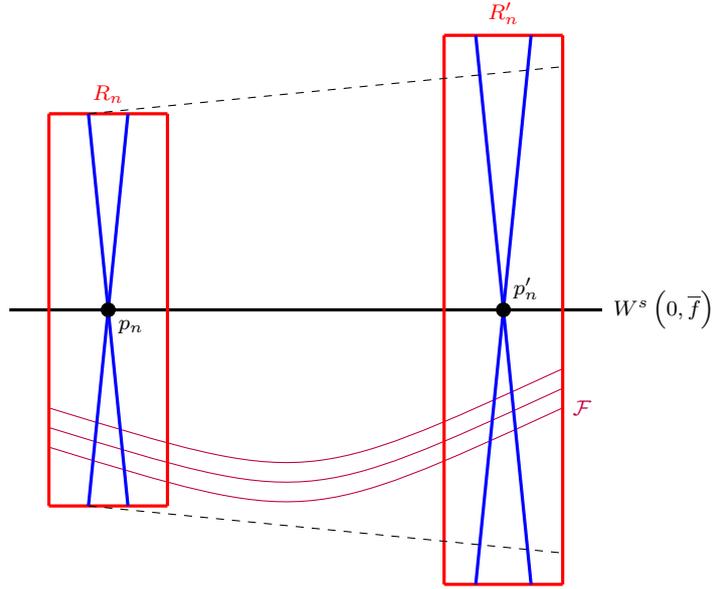
\begin{figure}[h]
   \begin{center}
    
   \begin{tikzpicture}[scale=5.2]
\draw[purple] (-0.65 ,-0.25) .. controls (0, -0.45) .. (0.65,-0.15);
\draw[purple] (-0.65 ,-0.35) .. controls (0, -0.55) .. (0.65,-0.25) node[right]{\scriptsize $\mathcal F$};

\draw[purple] (-0.65 ,-0.3) .. controls (0, -0.5) .. (0.65,-0.2);

\draw[black, very thick] (-0.75, 0) -- (.75,0) node[above, right]{\scriptsize $W^s\left(0, \overline f\right)$};


\draw[red, very thick] (-.65,-.5) -- (-.65, .5);
\draw[red, very thick] (-.65,.5) -- (-.35, .5) node[midway, above]{\scriptsize $R_n$};

\draw[red, very thick] (-.65,-.5) -- (-.35, -.5);

\draw[red, very thick] (-.35,-.5) -- (-.35, .5);

\draw[red, very thick] (.65,-.7) -- (.65, .7);
\draw[red, very thick] (.65,.7) -- (.35, .7) node[midway, above]{\scriptsize $R_n'$};

\draw[red, very thick] (.65,-.7) -- (.35, -.7);

\draw[red, very thick] (.35,-.7) -- (.35, .7);


\draw[blue, very thick] (-0.45,0.5) -- (-0.55, -0.5);

\draw[blue, very thick] (-0.45, - 0.5) -- (-0.55, 0.5);

\draw[blue, very thick] (0.43, - 0.7) -- (0.57, 0.7);

\draw[blue, very thick] (0.43,  0.7) -- (0.57, -0.7);


\draw[black, dashed]  (-0.55, 0.5) -- (0.65, 0.62);

\draw[black, dashed]  (-0.55, -0.5) -- (0.65, -0.62);


\filldraw[black] (-0.5,0) circle (.5pt) node[below right]{\scriptsize $p_n$};

\filldraw[black] (0.5,0) circle (.5pt) node[above right]{\scriptsize $p_n'$};



\end{tikzpicture}

    \caption{Illustration of rectangles $R_n$ and $R_n'$. The blue lines represent the boundaries of the unstable cones $K^u(p_n)\cap R_n$ and $K^u\left(p_n'\right)\cap R_n'$. The leaves of $\mathcal F$ connecting $\kappa$ and $H(\kappa)$ as in property (P2) must stay between the dashed lines, since $\kappa$ would remain in the left cone, which guarantees that property (P2) holds.}
    \label{Non-example Figure 1}

    \end{center}
\end{figure}

There is an $\epsilon>0$ such that if $\|f-A\|_{C^1}<\epsilon$, then $f$ satisfies the same cone condition as $A$ with respect to $K^s$ and $K^u$. Now, we may construct a sequence of $C^2$ diffeomorphisms $(\varphi_n)$, $\varphi_n:\mathbb T^2\rightarrow \mathbb T^2$, such that for some small $b>0$, we have
\begin{itemize}
\item $\varphi_n\restriction_{W^s_{\text{loc}}(0,A)}=\text{id}$
    \item $\|\varphi_n-\text{id}\|_{C^1}<\frac{1}{2}\epsilon$
     \item $\|D_x\varphi_n v\|\geq\left(1+b\right)\|v\|$ for all $v\in K^u(x)$ and $x\in R_n$ 
    \item $D_x\varphi_n=\text{Id}$ for all $x\in R_n'$ 
\end{itemize}

The third and fourth items are simultaneously possible since we constructed the rectangles $R_n$ and $R'_n$ in such a way that their sizes are proportional to $\left|p_n-p_n'\right|$.

Now, set $f_n:=\varphi_n\circ A$. Then, since $\|f_n-A\|_{C^1}<\epsilon$, we have that the collection of maps $\{f_n\}$ satisfies the same cone condition as $A$ but notice that $\left\{\|f_n\|_{C^2} \right\}$ is not uniformly bounded. 

Let us suppose that $\mathcal W^s\left(\overline f\right)$ is a $C^1$ foliation, and consider the holonomy map $H:\tau \rightarrow \tau'$. Since $H$ is $C^1$, by the Mean Value Theorem, for any segment $I\subset \tau$ containing $p$, we have 
\begin{equation}\label{MVT in example}
\text{length}\left(H\left(I\right)\right)\leq \|DH\|\cdot\text{length}(I).
\end{equation}
Let us choose $N$ so that 
\begin{equation}\label{Choice of large N in example}
    \left(1+b\right)^N\gg \|D\emb{H}\|.
\end{equation}

By the first property (P1) of the rectangle $R_N$, we can pick a segment $I\subset \tau$ so that $f_N\circ\cdots \circ f_1(I)$ is a curve connecting the base sides of $R_N$. In this case, the curve $f_N\circ\cdots \circ f_1(I)$ is contained in $K^u\left(p_N\right)\cap R_{N}$ and $\text{length}\left(f_N\circ\cdots \circ f_1(I)\right)\approx \left|p_N-p_N'\right|$.

For each $n\in\mathbb N$, denote $\tau_n:=f_n\circ\cdots\circ f_1(\tau)$ and $\tau'_n:=f_n\circ\cdots\circ f_1\left(\tau'\right)$. Notice that $f_N\circ \cdots \circ f_1\left(\mathcal W^s\left(\overline f\right)\right)$ is a $C^1$ foliation with leaves tangent to $K^s$. If we consider the holonomy map $H_N:\tau_N\rightarrow \tau_N'$ with respect to this foliation, then since its leaves are tangent to $K^s$, we must have that $H_N\left(f_N\circ\cdots \circ f_1(I)\right)\subset R_N'$, by Property (P2). Also, since $\tau_N$ and $\tau_N'$ are both tangent to $K^u$, then by the geometry of the cones, we have $\text{length}\left(f_N\circ\cdots \circ f_1(I)\right)\asymp \text{length}\left(H_N\left(f_N\circ\cdots \circ f_1(I)\right)\right)$ (see Figure \ref{Non-example Figure 2}). Precisely speaking, 
\begin{equation}\label{comparing transversal lengths}
    c\cdot\text{length}\left(f_N\circ\cdots f_1(I)\right)\leq \text{length}\left(H_N\left(f_N\circ\cdots \circ f_1(I)\right)\right),
\end{equation}
where $c>0$ is some small constant that only depends on the choice of cones and rectangles. We have
$$f_N\circ\cdots \circ f_1\left(H\left(I\right)\right)=H_N\left( f_N\circ\cdots \circ f_1\left(I\right)\right),$$
and hence, from (\ref{comparing transversal lengths}) and construction of $(\varphi_n)$, we have
\begin{align*}
&\text{length}\left(H\left(I\right)\right)=\eta^{-N}\text{length}\left(  f_N\circ\cdots \circ f_1\left(H\left(I\right)\right) \right)\\
&=\eta^{-N}\text{length}\left( H_N\left( f_N\circ\cdots \circ f_1(I)\right) \right)\geq c\eta^{-N}\text{length}\left(f_N\circ\cdots \circ f_1(I)\right)\\
&\geq c\eta^{-N}\left(1+b\right)^N\eta^{N}\text{length}\left(I\right)=c(1+b)^N\text{length}(I),
\end{align*}
but by choice of $N$ in (\ref{Choice of large N in example}), this contradicts (\ref{MVT in example}). 

\begin{figure}[h]
   \begin{center}
    
   \begin{tikzpicture}[scale=5.2]
\draw[purple] (-0.65 ,-0.25) .. controls (0, -0.05) .. (0.65,-0.15);
\draw[purple] (-0.65 ,-0.35) .. controls (0, -0.15) .. (0.65,-0.25) node[right]{\scriptsize $f_N\circ\cdots \circ f_1\left(W^s\left( \overline f\right)\right)$};

\draw[purple] (-0.65 ,-0.3) .. controls (0, -0.1) .. (0.65,-0.2);

\draw[black, very thick] (-0.75, 0) -- (.75,0) node[above, right]{\scriptsize $W^s\left(0, \overline f\right)$};


\draw[blue, dashed] (-0.4, 0.5) -- (-.6, -0.5);

\draw[blue, dashed] (-.6, 0.5) -- (-.4, -0.5);

\draw[blue, dashed] (0.4, 0.5) -- (.6, -0.5);

\draw[blue, dashed] (.6, 0.5) -- (.4, -0.5);

\filldraw[black] (-0.5,0) circle (.5pt) node[below right]{\scriptsize $p_N$};

\filldraw[black] (0.5,0) circle (.5pt) node[above right]{\scriptsize $p_N'$};


\draw[red, dashed] (-0.45,0.45) -- (0.47, 0.266);

\draw[red, dashed] (-0.45,-0.45) -- (0.44, -0.272 );


\filldraw[black] (-0.45,0.45) circle (.25pt);

\filldraw[black] (0.47, 0.266) circle (.25pt);

\filldraw[black] (-0.45,0.45) circle (.25pt);

\filldraw[black] (-0.45,-0.45) circle (.25pt);

\filldraw[black] (0.49, -0.262 ) circle (.25pt);


\draw[black, very thick] (-0.45,0.45) .. controls (-0.5, 0) .. (-0.45,-0.45);

\draw[black, very thick] (0.47, 0.266) .. controls (0.5, 0) .. (0.49, -0.262);


\end{tikzpicture}

    \caption{Illustration of $f_N\circ\cdots \circ f_1(I)$ and $H_N\left(f_N\circ\cdots \circ f_1(I)\right)$. These curves remain within the dashed blue lines since they are tangent to $K^u$. The leaves from $f_N\circ \cdots \circ f_1\left(W^s\left(\overline f \right)\right)$ connecting these curves remain within the red dashed lines since they are tangent to $K^s$. Using the geometry of these cones, one can establish (\ref{comparing transversal lengths}).}
    \label{Non-example Figure 2}

    \end{center}
\end{figure}
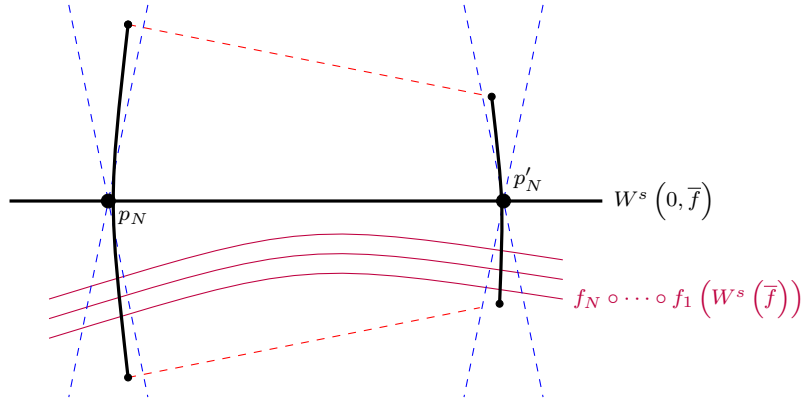

\section{The Stable Codimension 1 Case}\label{s: The Codimension 1 Case}
In this section we present an improved version of Theorem \ref{Main Foliation Theorem}. Let $d>1$ and suppose that $\overline f:=(f_n)$, $f_n:\mathbb T^d\rightarrow \mathbb T^d$, is a sequence of $C^1$ Anosov diffeomorphisms such that $\{f_n\}$ satisfies a common cone condition as in Section \ref{background and preliminaries}, but instead, we impose 
\begin{equation}
    \text{dim}(H_x)=d-1 \ \text{and} \ \text{dim}(V_x)=1.
\end{equation}
Then, the same stable manifold theory applies, which says that for each $x\in \mathbb T^d$, the stable set $W^s\left(x,\overline f\right)$ is the image of a $C^1$ immersion from $\mathbb R^{d-1}$ to $\mathbb T^d$ and that $\mathcal W^s\left(\overline f\right)$ is a $(d-1)$-dimensional continuous foliation of $\mathbb T^d$. Moreover, $E^s\left(x,\overline f\right)$ is a $(d-1)$-dimensional subspace of $T_x\mathbb T^d$ that varies continuously in $x$. We prove

\begin{theorem}\label{Smoothness of Splitting Codimensiona 1} As above, if $f_n\in \mathrm{Diff}^2\left(\mathbb T^d\right)$ for all $n\in\mathbb N$ and
\begin{equation}\label{e: C^2 norms bounded condition}
    \max\left\{\sup_{n}\left\|f_n\right\|_{C^2}, \sup_{n}\left\|f_n^{-1}\right\|_{C^2}\right\}<\infty,
\end{equation}
then the bundle $E^s\left(\overline f\right)$ is $C^1$. If $f_n\in \mathrm{Diff}^3\left(\mathbb T^d\right)$ for all $n\in\mathbb N$, and 
\begin{equation}\label{e: C^3 norms bounded condition}
    \max\left\{\sup_{n}\left\|f_n\right\|_{C^3}, \sup_{n}\left\|f_n^{-1}\right\|_{C^3}\right\}<\infty,
\end{equation}
then the bundle is $C^{1+\beta}$ for some $\beta\in (0,1)$.
\end{theorem}
As before, we obtain the following corollary. 
\begin{corollary}
    As above, (\ref{e: C^2 norms bounded condition}) implies that $\mathcal W^s\left(\overline f\right)$ forms a $C^1$ foliation of $\mathbb T^d$, and (\ref{e: C^3 norms bounded condition}) implies that $\mathcal W^s\left(\overline f\right)$ forms a $C^{1+\beta}$ foliation of $\mathbb T^d$ for some $\beta\in(0,1)$. 
\end{corollary}

\subsection{Improved Non-stationary Section Theorem}
While the proof of Theorem \ref{Smoothness of Splitting Codimensiona 1} is nearly identical to that of Theorem \ref{Smoothness of Splitting}, we need a slightly improved version of Theorem \ref{Torus Section Theorem}.

Let $M$ be a smooth compact Riemannian manifold and $K\subset \mathbb T^d\times M$ a compact set such that for some continuous function $\varphi: \mathbb T^d\rightarrow M$, we have $\text{graph}(\varphi)\subset \text{int}(K)$. Denote $\pi_M:\mathbb T^d\times M\rightarrow M$ to be the projection map and set $K_x:=\pi^{-1}_M(\{x\})\cap K$.
\begin{theorem}\label{Improved Section Theorem}
 Suppose that $(g_n)$, $g_n:\mathbb T^d\rightarrow \mathbb T^d$, is a sequence of $C^1$ diffeomorphisms such that 
 \begin{equation}
  \kappa:=  \sup_{n\in\mathbb{N}}\left\{ \sup_{x\in\mathbb T^d}\kappa_n(x)\right\}<\infty,
\end{equation} 
where $\kappa_n(x)=\left\|D_xg_n^{-1}\right\|$.
Furthermore, suppose $(h_n)$, $h_n:\mathbb T^d\times M\rightarrow M$, is a sequence of $C^1$ maps such that 
\begin{equation}
    h_n(K_x)\subset K_{g_n x}
\end{equation}
and 
\begin{equation}
   \sup_n \sup_{(x,z)\in K}\{\|\lambda_n(x,z)\|\}<1,
\end{equation}
where $\lambda_n(x,z):=D_{(x,z)}(\pi_M\circ h_n)$. Set
\begin{equation}
    F_n:\mathbb T^d\times M\rightarrow \mathbb T^d\times M, \ (x,z)\rightarrow (g_nx, h_n(x,z)).
\end{equation}
If $\{h_n\}$ is uniformly bounded in the $C^1$ topology and 
\begin{equation}\label{e: fiber and base contraction for skew product}
    \Delta:=\sup_{n\in\mathbb N}\left\{\sup\limits_{(x,z)\in K}\lambda_n\left(g_n^{-1}x,z\right)\kappa_n(x)\right\}<1,
\end{equation} 
then there is a $C^1$ map $\sigma^*:\mathbb T^d\rightarrow M$ such that 
\begin{equation}\label{e: Intersection of cones is graph}
    \bigcap_{n=1}^{\infty}F_1\circ\cdots\circ F_n\left(K\right)=\mathrm{graph}(\sigma^*).
\end{equation}
Furthermore, if we assume that the maps $h_n$ are $C^2$, the collection $\{h_n\}$ is uniformly bounded in the $C^2$ topology, and
\begin{equation}\label{e: holder fiber and base contraction for skew product}
    \Delta_
\beta:=\sup_{n\in\mathbb N}\left\{\sup\limits_{(x,z)\in K}\lambda_n\left(g_n^{-1}x,z\right)\kappa_n(x)^{1+\beta}\right\}<1,
\end{equation}
for some $\beta\in(0,1)$, then $\sigma^*$ is $C^{1+\beta}$.
\end{theorem}

\begin{proof}[Sketch of Proof.]
The argument is the same as that of Theorem \ref{Torus Section Theorem}, but without the convenience of writing the coordinates of $\mathbb T^2\times S^1$ and explicit structure of the set $K$.
\end{proof}

\subsection{Proof in the Stable Codimension 1 Case}
We now explain the proof of Theorem \ref{Smoothness of Splitting Codimensiona 1}.
\begin{proof}[Sketch of Proof of Theorem \ref{Smoothness of Splitting Codimensiona 1}.]
    The proof here is along the same line as the proof of Theorem \ref{Smoothness of Splitting} with an application of Theorem \ref{Improved Section Theorem} to the sequence of skew-product maps $(F_n)$ via
    \begin{equation}
        F_n:\mathbb T^d\times \text{Gr}_{d-1}\left(\mathbb R^d\right)\rightarrow \mathbb T^d\times \text{Gr}_{d-1}\left(\mathbb R^d\right), \ (x,l)\mapsto (g_nx, h_n(x,l))
    \end{equation}
over the set
    \begin{equation}
        K:=\left\{(x,l)\in \mathbb T^d\times \text{Gr}_{d-1}\left(\mathbb R^d\right) : l\subset K^s(x)\right\}. 
    \end{equation}
Here, $g_n=f_n^{-1}$ and $h_n(x,\cdot)$ is the action of $D_xf_n^{-1}$ on $\text{Gr}_{d-1}\left(\mathbb R^d\right)$, the space of $(d-1)$-dimensional subspaces of $\mathbb R^d$. 
    The arguments are the same and the estimate (\ref{e: Fiber and base contraction}) also holds in this context again from the cone conditions but now from the fact that each $K^s(x)$ is a cone with codimension $1$.
\end{proof}
\begin{remark}
    We note that with the given assumptions, (\ref{e: Fiber and base contraction}) does not necessary hold if each $H_x$ does not have codimension 1. In the stationary case, this dimension assumption is usually replaced with the so-called \textit{bunching condition} (see, for example, Section 2.2.4 of \cite{BC} for a discussion).
\end{remark}
\appendix
\section{Technical Lemmas}\label{Technical Lemmas}
 \begin{appendixlemma}\label{non-stationary contraction principle}
    Let $X$ be a compact metric space and $(\gamma_n)$ a sequence with $\gamma_n:X\rightarrow X$. If $\sup_{n}\text{Lip}(\gamma_n)<1$, then there is a $x^*\in X$ such that 
    $$\lim\limits_{n\rightarrow\infty}\gamma_1\circ\cdots\circ \gamma_n(x)=x^*$$
    for all $x\in X$
\end{appendixlemma}

\begin{proof}
See Theorem 3 in \cite{LY}.
\end{proof}

\begin{appendixlemma}\label{fiber contraction}
    Suppose that $X$ is a compact metric space and $\left(Y, \|.\|\right)$ is a Banach space. Suppose that $(\gamma_n)$, $\gamma_n:X\rightarrow X$, is a sequence of maps such that 
    \begin{equation}\label{base contraction rate in skew product}
    \sup_{n\in\mathbb N}\text{Lip}\left(\gamma_n\right)<1,
    \end{equation}
    and for each $x\in X$, we have a sequence of maps $\left(\psi_n^x\right)$, $\psi_n^x:Y\rightarrow Y$, such that
    \begin{equation}\label{fiber contraction rate in skew product}
        \sup_{n\in\mathbb N}\sup_{x\in X}\text{Lip}\left(\psi_n^x\right)<1.
    \end{equation}
    Consider the skew product maps $$\Phi_n:X\times Y\rightarrow X\times Y, (x,y)\mapsto \left(\gamma_n (x), \psi_n^x(y)\right).$$
    If there exists a $M>0$ such that
     \begin{itemize}
        \item[(1)]  $\left\|\psi_n^{x}(y)-\psi_n^{x'}(y)\right\|\leq M\left(1+\|y\|\right)$
        for all $x, x'\in X$ and $y\in Y$
        \item[(2)] For each $L'>0$, we have $\lim\limits_{\left|x-x'\right|\rightarrow0}\left\|\emb{\psi_n^x(y)- \psi_n^{x'}(y)}\right\|=0$ and this limit is uniform  in $y\in B_{L'}(0)$
    \end{itemize}
  then, there exists a $\left(x^*, y^*\right)\in X\times Y$ such that 
   $$\lim\limits_{n\rightarrow\infty}\Phi_1\circ \cdots\circ \Phi_n(x,y)=\left(x^*, y^*\right)$$
   for any $(x,y)\in X\times Y$.
\end{appendixlemma}
\begin{proof}
    This is a weaker version of Theorem 4 in \cite{LY}, since here, we assume $X$ is compact and condition (1) clearly implies condition (2) in Theorem 4 of \cite{LY}.
\end{proof}

\begin{remarksection}\label{equicontinuity replacement remark}
    Notice that condition (2) can be replaced with the assumption that for each $L'>0$, the family of maps $\left\{x\mapsto \emb{\psi_n^{x}\restriction_{y\in B_{L'}(0)}}\right\}$ is (uniformly) equicontinuous for each $n\in\mathbb N$.
\end{remarksection}

\section{Applications and Open Questions}\label{Applications}
We now explore some open questions regarding spectral properties of Sturmian Hamiltonians, and discuss how these problems are related to regularity results for stable foliations of non-stationary systems.
Consider the bounded self-adjoint operator $H_{\lambda, \alpha, \omega}: \ell^2(\mathbb{Z})\rightarrow \ell^2(\mathbb{Z}) $
via
$$[H_{\lambda, \alpha, \omega}u](n)= u(n+1)+u(n-1)+\lambda\chi_{[1-\alpha, 1)}\left(\omega+n\alpha \ (\text{mod} \ 1)\right)u(n),$$
where $\lambda>0$, $\alpha\in(0,1)$ is irrational, and $\omega\in S^1$. It is well-known (e.g. see Theorem 4.9  \cite{DF}) that the spectrum of $H_{\lambda, \alpha, \omega}$ is independent of $\omega$, so we will denote it by $\sigma_{\lambda, \alpha}$. 
\begin{theoremsection}[\cite{bist}]
    For any $\lambda>0$ and irrational $\alpha$, the spectrum $\sigma_{\lambda, \alpha}$ is a Cantor set of zero Lebesgue measure.
\end{theoremsection}

To analyze the dimension of $\sigma_{\lambda,\alpha}$, especially for small ranges of $\lambda>0$, the universal technique has been to analyze the dynamics of the so-called trace maps. 

\begin{definition}[Trace Maps]
    We define the \textit{trace maps} to be the family of maps given by 
    $$
T_k : \mathbb R^3 \to \mathbb R^3 , \quad \begin{pmatrix} x \\ y \\ z \end{pmatrix} \mapsto \begin{pmatrix} x U_k(y) - z U_{k-1}(y) \\ x U_{k-1}(y) - z U_{k-2}(y) \\ y \end{pmatrix} ,
$$
where $U_k(y)$ are the Chebyshev polynomials  of the second kind. That is, 
$$
 U_{-1} (y)  = 0 , \
 U_0 (y)  = 1 , \
 U_{k+1}(y)  = 2y U_k(y) - U_{k-1}(y).
$$
\end{definition}

    For $\lambda>0$, we consider the family of cubic surfaces given by
$$S_\lambda:=\left \{ (x, \ y, \ z)\in\mathbb{R}^3 : x^2+y^2+z^2-2xyz=1+\frac{\lambda^2}{4} \right\}.$$
Each surface $S_{\lambda}$ is invariant under each trace map and from here, we will consider the restriction of each $T_k$ to $S_\lambda$. This invariance is known as the Fricke-Vogt invariance.

\begin{theoremsection}[\cite{bist, D2000}]
     For $\lambda>0$ and $\alpha=[a_1,a_2,\ldots]\in(0,1)$ irrational, a real number $E$ belongs to $\sigma_{\lambda, \alpha}$ if and only if 
$$\left(T_{a_n}\circ\dots\circ T_{a_1}\left(\frac{E-\lambda}{2}, \frac{E}{2}, 1\right)\right)_{n\in\mathbb{N}}$$
is bounded.
\end{theoremsection}

In other words, if we consider the line 
$$L_{\lambda}:=\left\{\left(\frac{E-\lambda}{2}, \frac{E}{2}, 1\right) : E\in\mathbb{R}\right\}\subset S_{\lambda}$$
and set 
$$\mathcal W^s_\lambda(\alpha) := \left\{ \boldsymbol x \in S_\lambda : (T_{a_n}\circ \cdots \circ T_{a_1} (\boldsymbol x))_{n\in\mathbb{N}} \ \text{is bounded} \right\},$$
we have that the spectrum $\sigma_{\lambda,\alpha}$ is affine equivalent to $L_{\lambda}\cap \mathcal W^s_\lambda(\alpha)$.

From \cite{Ca}, we know that each trace map $T_k\restriction_{S_\lambda}$ is hyperbolic on its non-wandering set for all $\lambda>0$. This was verified in a small range of $\lambda$ in \cite{DG1} for the trace map $T_1\restriction_{S_\lambda}$, and an explicit cone field was constructed in a neighborhood of such non-wandering set to establish this hyperbolicity. 

Recently, a common invariant cone field was constructed for all trace maps \cite{GJK}. We wish to use similar methods to those of Theorem \ref{Main Foliation Theorem} to prove the following statement: 
\begin{problem}\label{NS Trace Map Foliation}
    Prove that there is a $\delta>0$ such that for $0<\lambda<\delta $ and any irrational $\alpha=[a_1,a_2,\dots]\in(0,1)$, the set $\mathcal W^s_\lambda(\alpha) $ is a collection of one-dimensional smooth curves that can be included into a $C^1$ foliation. 
\end{problem}

We now demonstrate how Problem \ref{NS Trace Map Foliation} can be useful in proving dimensional results about $\sigma_{\lambda,\alpha}$.

Let us consider two irrationals $\alpha=[a_1, a_2, \ldots]$ and $\alpha'=[a_2, a_3, \ldots]$. 
A $C^1$ foliation that contains the leaves of $W_{\lambda}^s\left(\alpha'\right)$ intersects $L_{\lambda}$ transversally, and hence
$$\sigma_{\lambda,\alpha}\simeq L_{\lambda}\cap \mathcal W_{\lambda}^s\left(\alpha\right)=T_{a_1}\left(L_{\lambda}\right)\cap \mathcal W_{\lambda}^s\left(\alpha'\right)$$ and $$\sigma_{\lambda,\alpha'}\simeq L_{\lambda}\cap \mathcal W_{\lambda}^s\left(\alpha'\right).$$ If the leaves of $W^s_{\lambda}\left(\alpha'\right)$ are smooth and part of a $C^1$ foliation, then the holonomy map 
$$\theta: L_{\lambda}\cap \mathcal W_{\lambda}^s\left(\alpha'\right)\rightarrow T_{a_1}\left(L_{\lambda}\right)\cap \mathcal W_{\lambda}^s\left(\alpha'\right)$$
would have a $C^1$ extension. This would imply that $\sigma_{\lambda,\alpha}$ and $\sigma_{\lambda,\alpha'}$ are diffeomorphic as Cantor sets, and hence have the same Hausdorff dimension. Through these techniques, we would be able to deduce 
\begin{problem}\label{first open}
    Suppose $\lambda\in(0,\delta)$ and $\alpha=[a_1, a_2, \ldots]$ and $\beta=[b_1, b_2, \ldots]$ are irrationals such that for some $k\in \mathbb{Z}$ and all large enough $i\in \mathbb{N}$, we have $b_{i+k}=a_i$. Prove that $\sigma_{\lambda, \alpha}$ and $\sigma_{\lambda,\beta}$ are diffeomorphic (as Cantor sets).  In particular, they have the same Hausdorff dimension.
\end{problem}

Hence, we would have that the map $\alpha\mapsto \text{dim}_H(\sigma_{\lambda, \alpha})$ would be invariant under the Gauss map $\mathcal G:[0,1]\setminus{\mathbb{Q}}\rightarrow [0,1]\setminus{\mathbb{Q}}$ via
$$\mathcal G(x)=\frac{1}{x}-\left \lfloor{\frac{1}{x}}\right \rfloor ,$$
and by ergodicity of $\mathcal G$, this would prove
\begin{conjecture}
    For $\lambda>0$ sufficiently small, $\alpha\mapsto \dim_{H}\sigma_{\lambda,\alpha}$ is $\alpha$-a.e. constant with respect to the Gauss measure. 
\end{conjecture}
This conjecture would immediately imply Conjecture \ref{c: almost sure constancy}, for small $\lambda$.
This problem was originally conjectured by J. Bellissard and was first solved for $\lambda\geq24$ in \cite{DG2015} and this result was expanded on in \cite{CQ}. However, these work use the methods of periodic approximations, which traditionally have not been applicable for all ranges of $\lambda$. In addition, Problem \ref{NS Trace Map Foliation} can be helpful in investigating the following conjecture:

\begin{conjecture}
  For any irrational $\alpha\in(0,1)$, we have
    $$\lim\limits_{\lambda\rightarrow 0} \dim_H\sigma_{\lambda, \alpha}=1.$$
\end{conjecture}
This problem has recently been solved for the case when $\alpha$ is an irrational of bounded type in \cite{L}, but is open for general $\alpha$. For the set $\mathbb S_0:=S_0\cap [-1,1]^3$, each restricted trace map $T_a\restriction_{\mathbb S_0}$ is semiconjugate to a linear hyperbolic map on $\mathbb T^2$. Using this fact, the techniques used in \cite{L} reduced to analyzing dimension results for a sequence of Anosov maps satisfying a common cone condition. This work influenced the choice of setting addressed in Theorem \ref{Main Foliation Theorem}, as we believe it can be useful in answering questions regarding the dimension of $\sigma_{\lambda,\alpha}$ when $\lambda\approx 0$. Some other questions one might tackle are the following:
\begin{problem}
    What can be said about the regularity of $\dim_H \sigma_{\lambda, \alpha}$ as a function of $\lambda>0$?
\end{problem}

\section*{Acknowledgments}
I would like to thank Anton Gorodetski for his guidance and support, and Victor Kelptsyn for thoroughly reading through a first draft of this paper and offering suggestions for improvement. I would also like to thank Amie Wilkinson for suggesting some references, and Jonathan DeWitt for helpful mathematical conversations. I am also thankful to both anonymous referees for their insight and suggestions to help improve the quality of the text. This project was supported by NSF grant DMS-2247966 (PI: A. Gorodetski).









\newcommand{\Addresses}{{
  \bigskip
  \footnotesize

  \textsc{Department of Mathematics, University of California, Irvine, CA 92697, USA}\par\nopagebreak
  \textit{E-mail address}: \texttt{lunaar1@uci.edu}

}}

\Addresses

\end{document}